\documentclass[12pt]{amsart}
\usepackage{amssymb}
\usepackage{graphicx,algorithm,algpseudocode,color}
\usepackage{csquotes}

\algnewcommand\algorithmicinput{\textbf{Input:}}
\algnewcommand\Input{\item[\algorithmicinput]}
\algnewcommand\algorithmicoutput{\textbf{Output:}}
\algnewcommand\Output{\item[\algorithmicoutput]}

\newtheorem{thm}{Theorem}[section]
\newtheorem{lem}[thm]{Lemma}
\newtheorem{cor}[thm]{Corollary}

\newtheorem{dfn}[thm]{Definition}
\newtheorem{exa}[thm]{Example}

\numberwithin{equation}{section}
\newcommand{\B}{\mathbf{B}}
\newcommand{\N}{\mathbb{N}}

\newcommand{\Pol}{\mathrm{Pol}}

\newcommand{\Z}{\mathbb{Z}}
\newcommand{\0}{\mathbf{0}}
\newcommand{\1}{\mathbf{1}}
\newcommand{\st}{\ |\ }
\newcommand{\vin}{\rotatebox[origin=c]{-90}{$\in$}}
\let\bar\overline
\newcommand{\CAB}{\mathcal{C}_{A,\mathbf{B}}}
\newcommand{\C}[2]{\mathcal{C}_{#1,#2}}

\title{On the number of clonoids}
\author{Athena Sparks}
\date{\today}
\address{Department of Mathematics\\
University of Colorado
Boulder\\ USA}
\email{athena.sparks@colorado.edu}
\thanks{This material is based upon work supported by the National Science Foundation under Grant No. DMS 1500254.}

\begin{document}

\begin{abstract}
A clonoid is a set of finitary functions from a set $A$ to a set $B$ that is closed under taking minors.
Hence clonoids are generalizations of clones. 
By a classical result of Post, there are only countably many clones on a $2$-element set. 
In contrast to that, we present continuum many clonoids for $A = B = \{0,1\}$.
More generally, for any finite set $A$ and any $2$-element algebra $\B$, we give the cardinality of the set of clonoids from $A$ to $\B$ that are closed under the operations of $\B$.
Further, for any finite set $A$ and finite idempotent algebra $\B$ without a cube term (with $|A|,|B|\ge 2$) there are continuum many clonoids from $A$ to $\B$ that are closed under the operations of $\B$; if $\B$ has a cube term there are countably many such clonoids.
\end{abstract}

\maketitle

\section{Introduction}

A clone on a set $D$ is a set of finitary operations on $D$ that contains all projections and is closed under composition of functions (see~\cite[page 97]{Lau} for the definition). 
In particular, clones are closed under the usual manipulations of permuting variables, identifying variables, and introducing dummy variables in functions.
For subsets $A,B$ of $D$, the restriction of a clone on $D$ to the functions from powers of $A$ into $B$ is not a clone anymore.
However, this restriction is still closed under the variable manipulations mentioned above.
More precisely, such a set of functions is closed under minors.
For $k\in\N$, let $[k]:=\{1,\dots,k\}$.

\begin{dfn}
Let $A, B$ be sets, $k\in \N$, and $f\colon A^k\rightarrow B$.
For $\ell\in\N$ and $\sigma:[k]\rightarrow[\ell]$, the function
\[f^\sigma:A^\ell\rightarrow B, (x_1,\dots, x_\ell)\mapsto f(x_{\sigma(1)},\dots, x_{\sigma(k)})\] 
is a \emph{minor} of $f$. 
\end{dfn} 

Sets of functions that are closed under minors have been investigated by Pippenger in~\cite{Pi:GMF}.
He developed a Galois theory for them and sets of pairs of relations that generalizes the classical Galois theory for clones.
These sets reappeared recently when Brakensiek and Guruswami classified Promise Constraint Satisfaction Problems (PCSP) on Boolean, symmetric, self-dual relational structures via polymorphisms between relational structures $\mathbb{A}$ and $\mathbb{B}$ of the same type in~\cite{BG:PCSP}.
Independently, these sets were used by Aichinger and Mayr to investigate equational theories of algebras in~\cite{AM:FGEC}.
Following the notion introduced in that last paper we define:

\begin{dfn}~\cite[Definition 4.1]{AM:FGEC}
Let $A$ be a set and $\B=(B, \mathcal{F})$ an algebra.
For a subset $C$ of $\bigcup_{n\in\N}B^{A^n}$ and $k\in \N$, we let $C_k:=C\cap B^{A^k}$. 
We call $C$ a \emph{clonoid} with \emph{source set} $A$ and \emph{target algebra} $\B$ if 
\begin{enumerate}
    \item $C$ is closed under taking minors, and 
    \item for all $k\in\N$, $C_k$ is a subalgebra of $\B^{A^k}$. 
\end{enumerate}
The set of all clonoids with source $A$ and target algebra $\B$  is denoted $\CAB$.
\end{dfn}

Note that every subset $C$ of $\bigcup_{n\in\N}B^{A^n}$ that is closed under taking minors is a clonoid with target algebra the set $(B, \emptyset)$. 
Further, every clone $C$ on a set $A$ is a clonoid with source set $A$ and target algebra $(A,C)$.
 
It is a well known result of Post that there are only countably many clones on a two element set~\cite[Theorem 3.1.1]{Lau}.
Janov and Mu{\v c}nik showed that there are continuum many clones on any finite set with three or more elements~\cite[Theorem 8.1.3]{Lau}.
In light of these results, one may ask whether the number of clonoids for fixed source $A$ and target $\B$ depends on the size of $A$ and $\B$. 
We will show that there are already continuum many clonoids with source and target of size $2$ (see Corollary~\ref{cor:main}). 
 
We introduce some more notation that will be needed in the following sections.
Let $A$ be a set and $\B=(B,\mathcal{F})$ an algebra.
For a set $F\subseteq\bigcup_{n\in\N}B^{A^n}$, the clonoid with source set $A$ and target algebra $\B$ generated by the functions in $F$ is denoted $\langle F\rangle_{\B}$.
If $\mathcal{F}=\emptyset$, then we simply write $\langle F\rangle$.
Let $P$ and $Q$ be a pair of $m$-ary relations on  $A$ and $B$ respectively. 
A function $f\colon A^k\rightarrow B$ is a \emph{polymorphism} of $(P,Q)$ if $f$ applied component-wise to any $k$-tuple of elements of $P$ is an element of $Q$. 
For a set of pairs of relations $R:=\{(P_i,Q_i)\:: i\in I\}$ 
on $A$ and $B$, the set of functions that are polymorphisms of all pairs of relations in $R$ is denoted $\Pol(R)$.
If $R$ contains only a single pair of relations $(P,Q)$, we write $\Pol(P,Q)$ instead. 
Note that if $f\in\Pol(R)$, then any minor of $f$ is in $\Pol(R)$.  
A clonoid $C$ with source set $A$ and target algebra $\B$ is \emph{finitely related} if there exists a finite set of pairs of finitary relations $R:=\{(P_i,Q_i)\:: 1\le i\le n\}$ on $A$ and $B$ such that $C= \Pol(R)$. 
It can be easily shown that any finitely related clonoid is the polymorphism clonoid of a single relation.
 
Let $X$ be the $k\times (2^k-1)$ matrix with columns $\{x,y\}^k\setminus\{\bar x\}$ where $\bar x=(x,x,\dots, x)$. 
A \emph{$k$-cube term} of $\B$ is a $(2^k-1)$-ary term $c$ in the operations of $\B$ such that 
\[c(X) = \bar{x}\]
for all $x,y$ in $B$ where $c$ is applied to every row of $X$.
A \emph{near-unanimity (NU) term} of $\B$ is an $k$-ary ($k\geq 3$) term $f$ in the operations of $\B$ which satisfies
\[f(y,x,x, \dots,x,x)=f(x,y,x,\dots,x,x)=\dots=f(x,x,x,\dots,x,y)=x\]
for all $x,y\in B$.
A \textit{Mal'cev term} of $\B$ is a ternary term $f$ in the operations of $\B$ which satisfies
\[f(y,y,x)=f(x,y,y)=x\]
for all $x,y\in B$.
Clearly, if $\B$ has an NU-term or a Mal'cev term, then $\B$ has a cube term. 
 
The following main result of this paper gives more precise information about the cardinality of clonoids with target algebras of size $2$.

\begin{thm}\label{thm:main}
 Let $\CAB$ denote the set of all clonoids with finite source $A$ ($|A|>1$) and target algebra $\B$ of size $2$. Then
\begin{enumerate}
\item \label{it:NU}
  $\CAB$ is finite iff $\B$ has an NU-term;
\item \label{it:Malcev}
  $\CAB$ is countably infinite iff $\B$ has a Mal'cev term but no majority term;
\item \label{it:set}
  $\CAB$ has size continuum iff $\B$ has neither an NU-term nor a Mal'cev term.  
\end{enumerate}
Moreover, in cases~\eqref{it:NU} and~\eqref{it:Malcev} all clonoids in $\CAB$ are finitely related.  
\end{thm}  

Following the case distinction of the theorem, we consider the size of $\CAB$ for an arbitrary finite $\B$ with an NU-term in Section~\ref{sec:NU}, for $\B$ with a cube term in Section~\ref{sec:cube}, and for $\B$ without cube term in Section~\ref{sec:nocube}.
In Section~\ref{sec:proof} we combine the results from these sections to prove Theorem~\ref{thm:main}.
The backward direction of (1) and the forward direction of (3) hold for arbitrary finite algebras $\B$ of size at least 2; our proofs of the others require that $\B$ is Boolean.
It is unknown if the forward direction of (1) holds for arbitrary finite algebras $\B$ of size at least 2, however, we know that the backward direction of (3) does not. 
An example of a target algebra $\B$ that has neither an NU-term nor a Mal'cev term where $\CAB$ is countably infinite is given in Example~\ref{exa:InfnotNUMalcev}. 
This example also shows that (2) does not hold for arbitrary target algebras. 
It is not known if there exists $\B$ with a Mal'cev term but no NU-term where $\CAB$ is finite.

The following theorem addresses the size of $\CAB$ for a finite idempotent algebra $\B$. 

\begin{thm}\label{thm:main2}
Let $A$ be a finite set and $\B$ a finite idempotent algebra with $|A|,|B|>1$. Then $\CAB$ has size continuum iff $\B$ has no cube term. 
\end{thm}

This follows immediately from Theorem~\ref{thm:cube} and Lemma~\ref{lem:idempotent}. 
It is unknown if this holds for arbitrary finite target algebras of size at least two. 

Given these results, the following example gives a target algebra $\B$ that has neither an NU-term nor a Mal'cev term where $\CAB$ is countably infinite.

\begin{exa}\label{exa:InfnotNUMalcev}
Let $A$ be a finite set and $\B_1$ and $\B_2$ be algebras of size 2 and type (2,3). 
The binary operation $t$ is interpreted in $\B_1$ as the projection onto to first coordinate and in $\B_2$ as the projection onto the second coordinate. 
The ternary operation $s$ in $\B_1$ is the Mal'cev operation $x-y+z\pmod 2$ and $s$ is the ternary majority operation in $\B_2$. 
Because of $t$, we see that $\B_1$ and $\B_2$ are independent \cite[Lemma 2.1]{F:TIOAUSF}; that is, the term operations of $\B_1\times\B_2$ are exactly the functions of the form
\begin{align*}
(B_1\times B_2)^k&\rightarrow B_1\times B_2\\
 ((x_1,y_1),\dots,(x_k,y_k))&\mapsto(g(x_1,\dots, x_k), h(y_1,\dots, y_k))
\end{align*}
for $k\in\B$ and $g,h$ arbitrary term functions of $\B_1$, $\B_2$, respectively. 
In particular, $\B_1\times \B_2$ is an idempotent algebra that has a 3-cube term but neither an NU-term nor a Mal'cev term.

By Theorem~\ref{thm:main} (2), there are countably infinitely many clonoids with target algebra $\B_1$.
Each clonoid $C$ in $\C{A}{\B_1}$ can be identified with a clonoid $\hat{C}$ in $\C{A}{\B_1\times\B_2}$ where 
\[\hat{C}:=\{\hat{f}:A^k\rightarrow (B_1\times B_2),\, (x_1,\dots, x_k)\mapsto (f(x_1,\dots, x_k),0)\st f\in C_k\}.\]
Hence there are infinitely many clonoids with source $A$ and target algebra $\B_1\times\B_2$. 
Therefore, by Theorem~\ref{thm:main2}, the number of clonoids with source $A$ and target $\B$ is countably infinite. 
\end{exa}
 
Pippenger showed that there are continuum many clonoids where the target algebra is the set $\{0,1\}$ with no operations \cite[Proposition 3.4 and following discussion]{Pi:GMF}. 
Theorem~\ref{thm:main}~\eqref{it:set} gives a alternate proof to this result. 
Since each clonoid with a target $\{0,1\}$ is also a clonoid with target $\{0,\dots, n\}$ for any $n\ge 1$, we immediately have the following:
\begin{cor}\label{cor:main}
 For all $m,n\ge 1$, there are continuum many clonoids with source $\{0,\dots,m\}$ and target $\{0,\dots,n\}$. 
\end{cor}

\section{NU-terms}\label{sec:NU}

In this section, we will show that there are only finitely many clonoids with a finite source $A$ and algebra $\B$ with an NU-term.
In particular, we show that each such clonoid is the polymorphism clonoid of a single pair of relations on $A$ and $B$.
We identify $A^{A^{n}}$ with $A^{|A|^{n}}$ and let $\Pi_A^{|A|^{n}}$ be the set of all $|A|^{n}$-ary projections on $A$.

\begin{thm}\label{thm:NU}
Let $A$ be a finite set of size greater than 1 and $\B$ a finite algebra with $n$-ary NU-term ($n\geq 3$).
Let $C$ be a clonoid with source $A$ and target $\B$.
Then $C = \Pol(\Pi_A^{|A|^{n-1}},C_{|A|^{n-1}})$.
Hence there are only finitely many such clonoids with source $A$ and target $\B$. 
\end{thm}  
\begin{proof}
Let $f\colon A^k\to B$. 
We claim that  
\begin{equation} \label{eq:fCA2}
 f\in C \text{ iff all } |A|^{n-1}\text{-ary minors of } f \text{ are in } C.
\end{equation}
This is equivalent to $C = \Pol(\Pi_A^{|A|^{n-1}},C_{|A|^{n-1}})$.

The forward direction of \eqref{eq:fCA2} is immediate from the definition of clonoids. 
For the reverse direction, note that the $k$-ary functions in $C$ form a subalgebra $C_k$ of $\B^{A^k}$.
By the Baker-Pixley Theorem~\cite{BP:PICR}, $C_k$ is uniquely determined by its projections onto the subsets of $A^k$ with $n-1$ or fewer elements. 
More precisely,
\begin{equation} \label{eq:BP}
 f\in C_k \text{ iff } \forall I\subseteq A^k\text{ with }|I|\le n-1,\; \exists g\in C_k\text{ so that } f|_I = g|_I.
\end{equation}
 
Let $Z$ be a matrix with $n-1$ rows whose columns are the $|A|^{n-1}$ tuples of $A^{n-1}$ in some order. 
For fixed $x_1,\dots,x_{n-1}\in A^k$, let $X$ denote the matrix with rows $x_1,\dots,x_{n-1}$ and $k$ columns.
Let $\sigma\colon [k]\to [|A|^{n-1}]$ such that the $i$-th column of $X$ is equal to the $\sigma(i)$-th column of $Z$.
With functions acting on the rows of the corresponding matrices, we then have  
\begin{equation} \label{eq:minor}
 f(X) = f^{\sigma}(Z).
\end{equation}
With~\eqref{eq:BP} and~\eqref{eq:minor} it follows that $f\in C_k$.
Thus~\eqref{eq:fCA2} and the theorem are proved.
\end{proof}

\section{Cube term}\label{sec:cube}

In this section, we will show that all clonoids with a finite source and a target algebra with a cube term, in particular, with a Mal'cev term, are finitely related.
We will also construct infinitely many clonoids for a fixed algebra of size 2 with a Mal'cev term.

\begin{thm}\label{thm:cube}
Let $A$ be a finite set and $\B$ a finite algebra with cube term. Then  each clonoid $C$ with source $A$ and target $\B$ is finitely related. Hence there are at most countably many such clonoids.
\end{thm}

\begin{proof}
Let $\text{Inv}(C)$ denote the set of all relational pairs on $A$ and $B$ preserved by $C$.
Let $\{(P_i,Q_i)\:: i\in \N\}$ be an enumeration of $\text{Inv}(C)$.
Then $C=\Pol(\text{Inv}(C))$ by the Galois Connection given in~\cite{Pi:GMF}.
Define $C_j:=\Pol(\{(P_i,Q_i)\::i\le j\})$.
Then $C_1 \supseteq C_2 \supseteq\cdots$ is a descending chain and 
\begin{equation} \label{eq:chain} \bigcap_{j\in\N}C_j=C.
\end{equation}
By Theorem 5.3 in \cite{AM:FGEC}, $\CAB$ satisfies the DCC. 
Hence there exists $m\in\N$ such that $C_m=C_n$ for all $n\ge m$.
By~\eqref{eq:chain}, $C=C_m$ and $C$ is finitely related. 
\end{proof}

Next we show that there actually are infinitely many clonoids with target any Mal'cev algebra of size $2$ without an NU-term.
By Post's classification of Boolean clones, the clone of each such algebra is contained in the clone of $ (\{0,1\},+,\0,\1)$, where $\0,\1$ are the unary constant functions. 
For algebras $\B$ and $\B'$, if the clone of $\B'$ is contained in the clone of $\B$, then $\CAB\subseteq\C{A}{\B'}$ for any set $A$. 
So it suffices to show the following:

\begin{lem}\label{lem:ek}
There exists infinitely many clonoids with source $A$ of size at least two and target algebra $\B= (\{0,1\},+,\0,\1)$.
\end{lem}
\begin{proof}
Let $0,1\in A$ and for $k\in\N$ define
\[ e_k\colon A^k\to \{0,1\}, x\mapsto\begin{cases} 1 & \text{if } x = (1,\dots,1), \\ 0 & \text{else.} \end{cases} \]
We will show that 
 
\[ \langle e_1 \rangle_\B \subset \langle e_2 \rangle_\B \subset \dots \]
is an infinite ascending chain of clonoids with target $\B$.
The idea for this example was used by Bulatov in~\cite{Bu:NFM} to construct countably many expansions of $(\Z_4,+)$.

It is enough to show that 
\begin{equation}\label{eq:ek}
e_k\ne\sum_{i=1}^{k-1}a_ie_i^{\sigma_i} \text{ for any }a_i\in\{0,1\}\text{ and }\sigma_i\colon[i]\rightarrow[k].
\end{equation}
For any $i<k$ and $\sigma_i\colon[i]\rightarrow[k]$, let the support of $e_i^{\sigma_i}$ be $\{x\in \{0,1\}^k\::e_i^{\sigma_i}(x)=1\}$.
Note that the support of $e_i^{\sigma_i}$  has even size for any $i<k$ and $\sigma_i\colon[i]\rightarrow[k]$.
Hence $\sum_{i=1}^{k-1}a_ie_i^{\sigma_i}$ has support of even size for all $a_i,\sigma_i$.
Since the support of $e_k$ is odd, \eqref{eq:ek} follows immediately.  
\end{proof}

\section{Without cube term}\label{sec:nocube}

In this section, we will show that there are continuum many clonoids with a finite source and finite idempotent target algebra without a cube term. 
Additionally, we will show there are continuum many clonoids with a finite source and Boolean target algebra without a cube term, or equivalently without an NU-term or Mal'cev term.

Let $A=\{0,1,\dots,d\}$ and $B=\{0,1,\dots,e\}$ for $d,e\geq 1$.
Define the following $n$-ary relations on $A$ and $B$, respectively, for all $n\in\N$:
\begin{align*}
P_n&:=\{(1,0,\dots,0),(0,1,0,\dots,0),\dots,(0,\dots,0,1)\}\subseteq A^n,\\
Q_n&:=\{0,1\}^n\setminus \{(1,\dots,1)\}\subseteq B^n.
\end{align*}
For $U\subseteq\N$, let $R_{U}:=\{(P_n,Q_n)\:: n\in U\}$.
Note that $\0$ preserves $R_{U}$ for any $U\subseteq\N$.

Define the following $k$-ary functions for all $k\in\N$:
\[f_k\colon A^k\rightarrow \{0,1\},x\mapsto 
  \left\{
    \begin{array}{ll}
      1 & \text{ if } x\in P_k,\\
      0 & \text{ otherwise}.
    \end{array}
  \right.\] 
For $U\subseteq\N$, let $F_U:=\{f_k\:: k\in U\}$.

We show some connections between these functions and relations that we need later.

\begin{lem} \label{lem:fkPl} \mbox{}
\begin{enumerate}
\item \label{it:fkPl}
 Let $k,n\in\N$.
 Then $f_k$ preserves $(P_n,Q_n)$ iff $k\neq n$.
\item \label{it:FU0}
 $\langle F_U \rangle \subseteq \Pol(R_{\bar{U}})$ for each $U\subseteq\N$ where $\bar{U}$ is the complement of $U$.
\end{enumerate}
\end{lem}

\begin{proof}
For~\eqref{it:fkPl}, we see that $f_k$ does not preserve $(P_k,Q_k)$ since
\[\begin{array}{cccccc}
1 & 0 & \cdots & 0 & \xrightarrow{f_k} & 1\\
0 & 1 & \cdots & 0 & \xrightarrow{f_k} & 1\\
\vdots & \vdots &  & \vdots & \vdots & \vdots\\
0 & 0 & \cdots & 1 & \xrightarrow{f_k} & 1\\
\vin & \vin & \cdots & \vin &  & \rotatebox[origin=c]{-90}{$\not\in$}\\
P_k & P_k & \cdots & P_k &  & Q_k.
\end{array}\]
Next assume $n\neq k$ and $x_1,\dots, x_k\in P_n$.
Let $M$ be the $n\times k$ matrix where the $j$th column is $x_j$ for $1\le j\le k$.
If $n<k$, then at least one row of $M$ must have at least two entries equal to 1.
Thus at least one entry of $f_k(M)$, the $n$-tuple obtained by applying $f_k$ to the rows of $M$, is 0. 
Hence $f_k(M)$ is in $Q_n$. 
If $n> k$, then at least one row of $M$ is all zeros. 
So at least one entry of $f_k(M)$ is 0 and $f_k(M)$ is in $Q_n$.
This concludes the proof of~\eqref{it:fkPl}.

Item~\eqref{it:FU0} is immediate from~\eqref{it:fkPl}. 
\end{proof}

\begin{lem}\label{lem:idempotent}
Let $A$ be a finite set, $\B$ a finite idempotent algebra without a cube term, and $|A|,|B|>1$. Then the number of clonoids from $A$ to $\B$ is continuum. 
\end{lem}

\begin{proof}
By \cite[Theorem 2.1]{MMM:FRCAACT} $\B$ must have cube term blocker.
That is, there exists a nonempty proper subset $V$ of $B$ such that 
\[T_n:=B^n\setminus (B\setminus V)^n\]
is a subuniverse of $\B$ for all $n$.
Without loss of generality, assume $0\in V$ and $1\in B\setminus V$. 
Thus $Q_n\subseteq T_n$.
The statement is immediate from the following claim:
\begin{equation}\label{eq:idFUFN}
\langle F_U\rangle_{\B}\cap F_{\N}= F_U\text{ for each }U\subseteq\N.
\end{equation}
The inclusion $\supseteq$ is clear. To prove the converse, let $U\subseteq\N\setminus\{1\}$ and $n\in\N$ such that $f_n\in\langle F_U\rangle_{\B}$. 
Then $f_n=\varphi(f_{k_1}^{\sigma_1},\dots, f_{k_m}^{\sigma_m})$ for some $m$-ary $\varphi$ in the clone of $\B$, $k_1,\dots,k_m\in U$, and maps $\sigma_i\colon [k_i]\rightarrow [n]$ for $1\le i\le m$.
If $n=k_i$ for some $i$, then $f_n\in F_U$. 

Assume toward a contradiction that $n\ne k_i$ for any $i$.
Let $a_1,\dots, a_n$ enumerate $P_n$. 
By Lemma~\ref{lem:fkPl},
\[f_{k_i}^{\sigma_i}(a_1,\dots, a_n)=:b_i\in Q_n\]
for each $1\le i\le m$.
Therefore we have
\begin{align*}
    f_n(a_1,\dots, a_n)
    &=\varphi(f_{k_1}^{\sigma_1},\dots, f_{k_m}^{\sigma_m})(a_1,\dots, a_n)\\
    &=\varphi(b_1,\dots, b_m)\\
    &\in T_n \text{ since $b_1,\dots,b_m\in T_n$ and $\varphi$ preserves $T_n$}.
\end{align*}
However $f_n(a_1,\dots, a_n)=(1,\dots, 1)\not\in T_n$. 
This contradiction completes the proof of~\eqref{eq:idFUFN}.
\end{proof}

Theorem~\ref{thm:main2} follows immediately from Theorem~\ref{thm:cube} and Lemma~\ref{lem:idempotent}.
Next we show that Lemma~\ref{lem:idempotent} generalizes to nonidempotent Boolean algebras.

By Post's classification of Boolean clones, each clone on $\{0,1\}$ without an NU-term or a Mal'cev term is contained in a nonidempotent clone generated by one of the following sets of operations:
\begin{enumerate}
\item $\{\wedge,\0,\1\}$ or $\{\vee,\0,\1\}$,
\item $\{\neg,\0\}$,   
\item $\{\rightarrow\}$ or $\{\not\rightarrow\}$.
\end{enumerate}
Thus there are 3 cases up to duality. 
We will show that for each case there are continuum many clonoids with source $A$ and corresponding target algebra $\B$ by variations of the proof of Lemma~\ref{lem:idempotent}.
From this it follows that for algebras with smaller clone of term operations (e.g., the set $(\{0,1\},\emptyset)$), there are continuum many clonoids as well.
Note that the maximal clones without a cube term on sets of size at least 3 are not explicitly know. 
Hence we do not know whether Lemma~\ref{lem:idempotent} generalizes to arbitrary nonidempotent algebras.

We begin proving the 3 cases with the case where $\B=(\{0,1\}, \land,\0,\1)$.
\begin{lem}\label{lem:and_case}
The number of clonoids with finite source $A$ and target algebra $\B=(\{0,1\}, \land,\0,\1)$ is continuum. 
\end{lem}

\begin{proof}
Let $\B=(\{0,1\},\land,\0,\1)$ and $\B'=(\{0,1\},\land)$. 
Note that the clone of $\B$ is the clone of $\B'$ with the addition of the constant maps $\0,\1$. 
Hence for any subset $U\subseteq \N$, we have $\langle F_U\rangle_\B=\langle F_U\rangle_{\B'}\cup\{\0,\1\}$.
By Lemma~\ref{lem:idempotent} there are continuum many Boolean clonoids of the form $\langle F_U\rangle_{\B'}$. 
\end{proof}

Now we prove the case where $\B=(\{0,1\},\neg,\0)$.

\begin{lem}\label{lem:neg_case}
The number of clonoids with finite source $A$ and target algebra $\B=(\{0,1\},\neg,\0)$ is continuum. 
\end{lem}

\begin{proof}
As in the proof of Lemma~\ref{lem:idempotent}, the statement is immediate from the following claim:
\begin{equation} \label{eq:FUFN} \langle F_U\rangle_{\B}\cap F_\N= F_U \text{ for each } U\subseteq\N\setminus\{1\}.
\end{equation}
The inclusion $\supseteq$ is clear. To prove the converse, let $U\subseteq\N\setminus\{1\}$ and $\ell\in\N$ such that $f_\ell\in\langle F_U\rangle_{\B}$. 
Then $f_\ell=f^\sigma_k$ or $f_\ell=\neg (f^\sigma_k)$ for some $k\in U$ and map $\sigma\colon [k]\rightarrow[\ell]$. 
In the former case, Lemma~\ref{lem:fkPl} yields $\ell = k$ and further $\ell\in U$.
To see that the latter cannot occur, let $m\in\N, m\neq\ell,$ and let $a=(1,0,\dots,0)\in P_m$. 
We have 
\begin{align*}
    f_\ell(a,\dots, a)
    &=\neg(f_k^\sigma)(\underbrace{a,\dots, a}_{\ell})\\
    &=\neg f_k(\underbrace{a,\dots, a}_{k})\\
    &=(\underbrace{1,\dots,1}_m) \text{ since $k\in U$, so $k> 1$}\\
    &\not\in Q_m.
\end{align*}
 Thus $f_\ell$ does not preserve $(P_m,Q_m)$. This contradicts Lemma~\ref{lem:fkPl} and completes the proof of~\eqref{eq:FUFN}.
\end{proof}

The final case, where $\B=(\{0,1\}, \not\rightarrow)$, is given in the following lemma. 

\begin{lem}\label{lem:notimplies_case} 
The number of clonoids with finite source $A$ and target algebra $\B=(\{0,1\}, \not\rightarrow)$ is continuum. 
\end{lem}
\begin{proof}
First we show
\begin{equation}\label{eq:FUd}
 \langle F_U\rangle_\B\subseteq\Pol(R_{\bar U}) \text{ for each } U\subseteq\N.
\end{equation}
By Lemma~\ref{lem:fkPl}, $\langle F_U\rangle\subseteq\Pol(R_{\bar U})$.
Assume $g,h\in\langle F_U\rangle_\B$ of arity $k$ preserve $(P_n,Q_n)$ and let $a_1,\dots,a_k\in P_n$. 
Let $d := g \not\rightarrow h$. 
Then we have 
\[d(a_1,\dots,a_k) = g(a_1,\dots, a_k) \wedge (\neg h(a_1,\dots,a_k)).\]
Since $g$ preserves $(P_n,Q_n)$, there must be at least one zero entry in $g(a_1,\dots, a_k)$.
Thus $d(a_1,\dots, a_k)\in Q_n$. 
Hence~\eqref{eq:FUd} is proved.

Let $U,V\subseteq\N$ such that $U\ne V$.
We claim that
\begin{equation}
 \langle F_U\rangle_\B\neq\langle F_V\rangle_\B. 
\end{equation}
Without loss of generality, assume there exists $n\in U\setminus V$.
From~\eqref{eq:FUd}, we have that $\langle F_V\rangle_\B$ preserves $(P_n,Q_n)$.
Since $n\in U$, we have $f_n\in F_U$ and thus $\langle F_U\rangle_\B$ does not preserve $(P_n,Q_n)$ by Lemma~\ref{lem:fkPl}.
Therefore $\langle F_U\rangle_\B\ne \langle F_V\rangle_\B$.
\end{proof}

\section{Proof of Main Theorem}\label{sec:proof}

In this section, we combine the results from the previous sections to give a proof of Theorem~\ref{thm:main}. 

\begin{proof}[Proof of Theorem ~\ref{thm:main}]
The reverse direction of \eqref{it:NU} follows immediately from Theorem~\ref{thm:NU}. 

To prove the reverse direction of \eqref{it:Malcev}, assume $\B$ has a Mal'cev term but no majority term. 
Then by Theorem~\ref{thm:cube}, $\CAB$ is at most countably infinite.
Since $\B$ has no majority term, by Post's classification, the clone of $\B$ is contained in the clone of $\B':=(\{0,1\},+,\0,\1)$.
In Lemma~\ref{lem:ek} we show that there are infinitely many clonoids in $\C{A}{\B'}$.
Since $\C{A}{\B'}\subseteq\CAB$, there are countably many clonoids in $\CAB$. 

Now assume $\B$ has neither an NU-term nor a Mal'cev term. 
As mentioned in the beginning of Section~\ref{sec:nocube}, it follows from Lemmas~\ref{lem:and_case}, \ref{lem:neg_case}, and \ref{lem:notimplies_case} and their duals that there are continuum many clonoids with target $\B$.
This proves the reverse direction of \eqref{it:set}.

The forward directions of \eqref{it:NU}, \eqref{it:Malcev}, and \eqref{it:set} follow because the cases are mutually exclusive and cover all possibilities.
\end{proof}

\section*{Acknowledgments}
The author thanks Peter Mayr for discussions on the material in this paper and the anonymous referee for their comments and asking a question that led to Theorem 1.4.

\end{document}